\documentclass[preprint, 11pt, english]{elsarticle}
\usepackage{amsmath}
\usepackage{latexsym, amssymb}
\usepackage{amsthm}
\usepackage{txfonts,xcolor}
\usepackage{hyperref}

\newtheorem{thm}{Theorem}[section] 

\newtheorem{cor}[thm]{Corollary}

\newtheorem{lem}[thm]{Lemma}

\theoremstyle{definition}
\newtheorem{rem}[thm]{Remark}
\newtheorem{exmpl}[thm]{Example}

\newcommand\operA[2]{{\if!#2!\operatorname{#1}\else{\operatorname{#1}_{#2}^{\phantom{I}}}\fi}} 

\newcommand\Cref[1]{{Corollary~\ref{#1}}}%

\def\tr{{\operatorname{Tr}}}

\def\norm{{\operatorname{Norm}}}

\newcommand{\Trace}[1][]{\if!#1!\operatorname{Tr}\else{\operatorname{Tr}_{#1}^{\phantom{I}}}\fi} 

\long\def\forget#1\forgotten{{}} %

\def\({\left(}
\def\){\right)}

\newif\iffurther
\furtherfalse

\newif\ifXY 
\XYtrue     
\ifXY

\input xy
\input xyidioms.tex
\usepackage{xy}
\xyoption{all} %
\fi 

\usepackage{babel}

\makeatletter
\def\ps@pprintTitle{%
	\let\@oddhead\@empty
	\let\@evenhead\@empty
	\def\@oddfoot{}%
	\let\@evenfoot\@oddfoot}
\makeatother

\begin{document}

\begin{frontmatter}

\title{Points with Commuting Coordinates over Division Rings}
\author{Adam Chapman}
\ead{adam1chapman@yahoo.com}
\address{Department of Computer Science, Academic College of Tel-Aviv-Yaffo, Rabenu Yeruham St., P.O.B 8401 Yaffo, 6818211, Israel}
\author{Solomon Vishkautsan}
\ead{wishcow@gmail.com}
\address{Department of Computer Science, Tel-Hai University of Kiryat Shmona in the Galilee, Upper Galilee, 12208, Israel}

\begin{abstract}

We study properties concerting polynomials in $D[x_1,\dots,x_n]$ and the set $D_c^n$ of points over a division ring $D$ with commuting coordinates.
Alon and Paran showed that for $D=\mathbb{H}$, this set parametrizes the set of maximal left ideals of $D[x_1,\dots,x_n]$.
Here we show that any root of $f$ in $D_c^n$ is also a root of its (reduced) norm $\norm(f)$, and that if a point in $D_c^n$ is a fixed point of an $n$-tuple $T=(f_1,\dots,f_n)$ of polynomials in $n$ variables, then it is a fixed point of $T^{\circ m}$ for any positive integer $m$.
\end{abstract}


\end{frontmatter}

\section{Introduction}

Given a division ring $D$, we define $D[x_1,\dots,x_n]$ to be $D \otimes_F F[x_1,\dots,x_n]$ where $Z(D)=F$.
Each polynomial $f \in D[x_1,\dots,x_n]$ can be written as the sum of monomials $c_{\vec{d}} \,x_1^{d_1} \dots x_n^{d_n}$ where $\vec{d}=(d_1,\dots,d_n)\in I$ and $I$ is a finite set of $n$-tuples of non-negative integers.
The substitution rule of $\vec{a}=(a_1,\dots,a_n)\in D^n$ in $f$ is thus $$f(\vec{a})=\sum_{\vec{d}\in I} c_{\vec{d}} \, a_1^{d_1}\dots a_n^{d_n}.$$
Similarly, one can substitute an $n$-tuple of polynomials $g_1,\dots,g_n\in D[x_1,\dots,x_n]$ in $f$ by 
$$f(g_1,\dots,g_n)=\sum_{\vec{d}\in I} c_{\vec{d}} \, g_1^{d_1}\dots g_n^{d_n}.$$
Fixing $\vec{a}\in D^n$, the substitution map $f\mapsto f(\vec{a})$ is additive but not multiplicative, and thus, not a ring homomorphism from $D[x_1,\dots,x_n]$ to $D$. For instance, if $f,g,h \in \mathbb{H}[x]$ and $h=fg$ where $f=ix$ and $g=jx$, then $f(i)=-1$, $g(i)=-ij$ and thus $f(i)g(i)=ij$, but $h(i)=iji^2=-ij$. And yet, substitution provides information about the polynomials, e.g., if $f(a)=0$ for $a\in D$ and $f\in D[x]$, then $f=g(x-a)$ for some $g\in D[x]$ (a classical result of Wedderburn \cite{Wedderburn:1921}). This phenomenon holds true also when we replace $D$ with an octonion algebra $A$, see \cite{Chapman:2020b}.

The special set of points $D_c^n=\{ \vec{a} \in D^n : a_k a_\ell=a_\ell a_k \forall k,\ell\}$ showed up to be of certain importance in the special case of $D=\mathbb{H}$. In particular, it was shown in \cite{AlonParan:2021} that $L$ is a maximal left ideal of $\mathbb{H}[x_1,\dots,x_n]$ if and only if $L=\langle x_1-a_1,\dots,x_n-a_n\rangle$ for some $\vec{a}\in D_c^n$, and in \cite{AlonParan:2024} that if $f\in \mathbb{H}[x_1,\dots,x_n]$ vanishes at all $\vec{a} \in \mathbb{H}_c^n$ at which $g\in \mathbb{H}[x_1,\dots,x_n]$ vanishes, then $f$ vanishes at all $\vec{a} \in \mathbb{H}^n$ at which $g\in \mathbb{H}[x_1,\dots,x_n]$ vanishes.

Our goal here is to prove certain properties the points in $D_c^n$ satisfy. In Section~\ref{sec:root-of-left-multiple} we prove that if $f \in D[x_1,\dots,x_n]$ vanishes at $\vec{a}\in D_c^n$ then $h=gf$ vanishes at $\vec{a}$ for any $g\in D[x_1,\dots,x_n]$. 
As a result, any root of $f$ in $D_c^n$ is also a root of $\norm(f)$ when $D$ is a central division ring, generalizing a result from \cite{ChapmanMachen:2017} that covers the case of $n=1$. We show in Section~\ref{sec:oct-roots} that the last statement also extends to octonion algebras, generalizing a result from \cite{Chapman:2020} that covers the case of $n=1$. 

Motivated by classical algebraic geometry and generalizations of the nullstellensatz to the quaternions in \cite{AlonParan:2021,AlonParan:2024, AlonChapmanParan:2025}, we look at the zero sets of systems of polynomials in Section~\ref{sec:zeros-sets}. While the classical property that a zero set of a system of polynomials equals the zero set of the ideal it generates fails for general division rings, we prove that this equality holds if we restrict our evaluation to points in $D_c^n$ and consider left ideals.

Finally, in Section~\ref{sec:dynamics}, we apply these algebraic evaluation properties to discrete algebraic dynamics over division algebras, generalizing results from \cite{ChapmanVishkautsan:2021} to the multivariate case. Let $T=(f_1,\dots,f_n)$ be an $n$-tuple of polynomials in $D[x_1,\dots,x_n]$ and $T(\vec{a})=(f_1(\vec{a}),\dots,f_n(\vec{a}))=\vec{a}$ for $\vec{a} \in D_c^n$. The iterate $T^{\circ m}$ is defined recursively by $T^{\circ 1}=T$ and $T^{\circ m}=(f_1(T^{\circ (m-1)}),\dots,f_n(T^{\circ (m-1)}))$. We show that if a point $\vec{a} \in D_c^n$ commutes with its orbit (see Section~\ref{sec:dynamics} for the definition), non-commutative evaluation behaves nicely, i.e.\ commutes with iteration. In particular, if $\vec{a}$ is a fixed point of $T$, it remains a fixed point for all iterates of $T$.


\section{Roots of Multivariate Polynomials}\label{sec:root-of-left-multiple}

We start with general division rings, and later restrict to central division rings (see \cite{GordonMotzkin1965} and \cite{GilleSzamuely2006} for background).

\begin{thm} \label{thm:root-of-left-multiple}
Let $D$ be a division ring and $f \in D[x_1,\dots,x_n]$.
Then, for $\vec{a}\in D_c^n$, if $\vec{a}$ is a root of $f$, then it is a root of $h=gf$ for any $g\in D[x_1,\dots,x_n]$.
\end{thm}

\begin{proof}
Write $f=\sum_{\vec{d} \in I} b_{\vec{d}} x_1^{d_1}\dots x_n^{d_n}$ and $g=\sum_{\vec{e} \in J} c_{\vec{e}} x_1^{e_1}\dots x_n^{e_n}$. So, 
$$h=\sum_{\vec{d} \in I} \sum_{\vec{e} \in J} c_{\vec{e}} b_{\vec{d}} x_1^{d_1+e_1} \dots x_n^{d_n+e_n}.$$
Suppose $f(\vec{a})=0$. Then, 
\begin{eqnarray*}
h(\vec{a}) & = & \sum_{\vec{d} \in I} \sum_{\vec{e} \in J} c_{\vec{e}} b_{\vec{d}} a_1^{d_1+e_1} \dots a_n^{d_n+e_n}\\
&=&\sum_{\vec{e} \in J} c_{\vec{e}} \left(\sum_{\vec{d} \in I} b_{\vec{d}} a_1^{d_1} \dots a_n^{d_n}\right) a_1^{e_1} \dots a_n^{e_n}\\
&=&\sum_{\vec{e} \in J} c_{\vec{e}} (f(\vec{a})) a_1^{e_1} \dots a_n^{e_n}=0.
\end{eqnarray*}

\end{proof}

Suppose now that $D$ is a division ring of finite dimension $m$ over its center $F$. Then, $D$ contains a maximal subfield $K$, and $D\otimes_F K \cong M_m(K)$. This gives rise to the (reduced) norm form $\norm : D \rightarrow F$ that maps each $t\in D$ to the determinant of its embedding $T$ in $M_m(K)$.
The (reduced) trace $\tr(t)$ of $t$ is similarly defined as $\tr(T)$.
More generally, each such $t$ gives rise to a characteristic polynomial
$p_t(\lambda)=|\lambda I-T|=\lambda^m+\varphi_1(t) \lambda^{m-1}+\dots+\varphi_{m-1}(t)\lambda+\varphi_m(t)$ where each $\varphi_r$ is a homogeneous polynomial form of degree $r$ from $D$ to $F$, and in particular $\varphi_1(t)=-\tr(t)$ and $\varphi_m(t)=(-1)^m \norm(t)$.
These maps do not depend on the choice of maximal subfield $K$ of $D$, and are homogeneous polynomial forms of degree $m$ for any choice of basis for $D$ as an $F$-vector space. This map extends to $D[x_1,\dots,x_n]$, and each $f\in D[x_1,\dots,x_n]$ is thus mapped to $\norm(f)\in F[x_1,\dots,x_n]$, and in particular $f|\norm(f)$, for $(-1)^{m+1}\norm(f)=(f^{m-1}+\varphi_1(f)f^{m-2}+\dots+\varphi_{m-1}(f))f$. When $D$ is a quaternion algebra, $\norm(f)=\bar{f}f$, where $\bar{z}$ is the (symplectic) conjugate of $z$ in $D$.

\begin{cor}
If $D$ is a central division $F$-algebra and $f\in D[x_1,\dots,x_n]$, then any root $\vec{a} \in D_c^n$ of $f$ is also a root of $\norm(f)\in F[x_1,\dots,x_n]$.
\end{cor}

This extends a result from \cite{ChapmanMachen:2017} where the case of $n=1$ was treated. Assuming that the roots of $\norm(f)$ in $D_c^n$ can be found, one can use their conjugacy class to try to retrieve the roots of $f$ in $D_c^n$. However, the only case where this results in a neat linear equation is when $D$ is a quaternion algebra and $n=1$.
Note that this corollary and its preceding theorem apply only to roots in $D_c^n$. The following easy example shows that it does not extend to general roots in $D^n$:

\begin{exmpl}
Consider $f=x+y-i-j\in \mathbb{H}[x,y]$. Then $\norm(f)=x^2+2xy-y^2+2$.
Now, $f(i,j)=0$ whereas $\norm(f)(i,j)=2ij\neq 0$.
\end{exmpl}

\section{Roots in octonion algebras} \label{sec:oct-roots}

Here, suppose that $A$ is an octonion algebra over a field $F$ and $A[x_1,\dots,x_n]$ is defined as $A\otimes_F F[x_1,\dots,x_n]$.
Recall that given a quaternion algebra $Q$ with symplectic involution $z\mapsto \bar{z}$, an octonion algebra $A=Q \oplus Q\ell$ is defined by $(q+r\ell)(s+t\ell)=qs+c \bar{t}r+(tq+s\bar{r})\ell$ where $c$ is a fixed element in $F^\times$.
When $\operatorname{char}(F)\neq 2$, $Q=(a,b)_{2,F}=F\langle i,j : i^2=a, j^2=b, ji=-ij \rangle$ for some $a,b \in F$, and then $A$ is denoted by $(a,b,c)_F$. 
When $\operatorname{char}(F)=2$, $Q=[a,b)_{2,F}=F \langle i,j : i^2+i=a, j^2=b, ji=ij+j \rangle$ for some $a\in F$ and $b\in F^\times$, and then $A$ is denoted by $[a,b,c)_F$ (see for example \cite{ElduqueVilla:2005}). The symplectic involution extends from $Q$ to $A$ by $\overline{q+r\ell}=\overline{q}-r\ell$, and it gives rise to the linear transformation $\tr(z)=\bar{z}+z$ and the quadratic form $\norm(z)=\bar{z}z$. This norm form is a quadratic 3-fold Pfister form, and it is hyperbolic if and only if $A$ is split, or equivalently, anisotropic if and only if $A$ is a division algebra.
The norm form extends of course to $A[x_1,\dots,x_n]$.

We define $\vec{a}\in A_c^n$ to be an $n$-tuple $\vec{a}=(a_1,\dots,a_n)$ in $A^n$ where $a_1,\dots,a_n$ belong to a single quadratic field extension $K$ of $F$ inside $A$. 

\begin{rem}
Note that this definition coincides with the condition that $a_k a_\ell=a_\ell a_k$ for all $k,\ell \in \{1,\dots,n\}$ when $\operatorname{char}(F)\neq 2$, but when  $\operatorname{char}(F)=2$, $A$ contains also a purely inseparable bi-quadratic extension of $F$, so requiring that the slots commute is not strong enough for our needs.
\end{rem}

Each polynomial $f\in A[x_1,\dots,x_n]$ can be written as $$f=\sum_{\vec{d}\in I} c_{\vec{d}} x_1^{d_1} x_2^{d_2}\dots x_n^{d_n},$$ and the substitution is set to be $f(\vec{a})=\sum_{\vec{d}\in I} c_{\vec{d}} (a_1^{d_1}(a_2^{d_2} (\dots (a_n^{d_n})\dots))$.

The octonion algebra $A$ is endowed with a symplectic involution $z\mapsto \bar{z}$, and $\norm(z)=\bar{z}z$ defines a quadratic 3-fold Pfister form from $A$ to $F$.
This map extends to $A[x_1,\dots,x_n]$, and for any $f\in A[x_1,\dots,x_n]$, we have $\norm(f)=\bar{f}f\in F[x_1,\dots,x_n]$.

\begin{thm}
For $\vec{a}\in A_c^n$, if $\vec{a}$ is a root of $f$, then it is a root of $\norm(f)$.
\end{thm}

\begin{proof}
Write $f=\sum_{\vec{d} \in I} b_{\vec{d}} x_1^{d_1}\dots x_n^{d_n}$. Then, $\bar{f}=\sum_{\vec{e} \in I} \overline{b_{\vec{e}}} \, x_1^{e_1}\dots x_n^{e_n}$. So, $$\norm(f)=\sum_{\vec{d} \in I} \sum_{\vec{e} \in I} \overline{b_{\vec{e}}} \, b_{\vec{d}} \, x_1^{d_1+e_1} \dots x_n^{d_n+e_n}.$$
Suppose $f(\vec{a})=0$. Then, 
\begin{eqnarray*}
\norm(f)(\vec{a}) & = & \sum_{\vec{d} \in I} \sum_{\vec{e} \in I} \left(\overline{b_{\vec{e}}}\, b_{\vec{d}}\right) \left( a_1^{d_1+e_1} \dots a_n^{d_n+e_n}\right)\\
&=&\sum_{\vec{e} \in I} \overline{b_{\vec{e}}} \left(\sum_{\vec{d} \in I} b_{\vec{d}} \, a_1^{d_1+e_1} \dots a_n^{d_n+e_n}\right)\\
&=&\sum_{\vec{e} \in I} \overline{b_{\vec{e}}} \left(\sum_{\vec{d} \in I} b_{\vec{d}} \, a_1^{d_1} \dots a_n^{d_n} a_1^{e_1} \dots a_n^{e_n}\right)\\
&=&\sum_{\vec{e} \in I} \overline{b_{\vec{e}}} (f(\vec{a}) a_1^{e_1} \dots a_n^{e_n})=0.
\end{eqnarray*}
For the step from the first line to the second, we use \cite[Lemma 1.3.3]{SpringerVeldkamp}: if $\vec{d}\neq \vec{e}$, then in the first line, the two terms $\left(\overline{b_{\vec{e}}} b_{\vec{d}}\right) \left( a_1^{d_1+e_1} \dots a_n^{d_n+e_n}\right)$ and $\left(\overline{b_{\vec{d}}} b_{\vec{e}}\right) \left( a_1^{d_1+e_1} \dots a_n^{d_n+e_n}\right)$ appear, and together they make $\tr(\overline{b_{\vec{e}}} b_{\vec{d}}) \left( a_1^{d_1+e_1} \dots a_n^{d_n+e_n}\right)$, and by this lemma, it is equal to the sum of the two terms $\overline{b_{\vec{e}}} \left( b_{\vec{d}} a_1^{d_1+e_1} \dots a_n^{d_n+e_n}\right)$ and $\overline{b_{\vec{d}}} \left( b_{\vec{e}} a_1^{d_1+e_1} \dots a_n^{d_n+e_n}\right)$.
\end{proof}

The following example shows that the last theorem does not hold true if instead of requiring that $a_1,\dots,a_n$ belong to a quadratic extension of the center, we require that they simply commute.

\begin{exmpl}
Suppose $F$ is a field of $\operatorname{char}(F)=2$, and $a,b,c \in F$ are such that $A=[a,b,c)_{2,F}$ is an octonion division algebra over $F$ with generators $i,j,\ell$. In particular, $i^2+i=a$, $j^2=b$ and $\ell^2=c$.
Now, $j$ and $\ell$ commute.
Consider $f=i+\frac{a}{bc}i(j\ell)xy+x+j$.
Then $f(j,\ell)=0$. However,
$\norm(f)=a+x+x^2+b+\frac{a}{bc}x^2y^2$, and so $\norm(f)(j,\ell)=a+j+b+b+a=j\neq 0$.
\end{exmpl}

\section{Zero sets of polynomials} \label{sec:zeros-sets}

In classical algebraic geometry (e.g., see \cite[Chapter I, \S1]{hartshorne}), for a given set of polynomials $S\subseteq k[x_1,\ldots,x_n]$ for an algebraically closed field $k$, we define 
\[\mathcal{V}(S) = \{P\in \mathbb{A}^n_k \,|\, f(P)=0 \text{ for all } f\in S\},\]
where $\mathbb{A}^n_k$ is the $n$-dimensional affine space over the field $k$. A basic property of this operation is that 
\begin{equation}\label{eq:zero-set}
\mathcal{V}(S)=\mathcal{V}(I)	
\end{equation}
where $I$ is the ideal generated by $S$ in $k[x_1,\ldots,x_n]$. 

We cannot expect property \eqref{eq:zero-set} to hold for general division rings. For example this fails for the real quaternions $\mathbb{H}$ already for $n=1$; indeed, $i$ is a root of $x-i$ but is not a root of $(x-i)(x-j)$, despite the latter belonging to the (two-sided) ideal generated by the former. 

We will show that a generalization for \eqref{eq:zero-set} to the division algebra setting is to consider ``central" points in $D^n_c$ and left-ideals. Let $S\subseteq D[x_1,\ldots,x_n]$, then
\[\mathcal{V}_c(S) = \{\vec{a}\in D^n_c \,|\, f(\vec{a})=0 \text{ for all } f\in S\}.\]

An easy corollary for Theorem~\ref{thm:root-of-left-multiple} is the following:
\begin{cor}
	Let $I$ be the left ideal generated by $S\subseteq D[x_1,\ldots,x_n]$, then $\mathcal{V}_c(S)=\mathcal{V}_c(I)$. 
\end{cor}

\begin{proof}
	Clearly, $\mathcal{V}_c(I)\subseteq \mathcal{V}_c(S)$. Let $\vec{a}\in \mathcal{V}_c(S)$, and let $f_1,\ldots,f_k\in S$ and $g_1,\ldots,g_k \in D[x_1,\ldots,x_n]$. Then by Theorem~\ref{thm:root-of-left-multiple} we know $\vec{a}$ is also a root of $\sum_{i=1}^k g_if_i$.
\end{proof}

\begin{rem}
The contents of this section fit in well with the Alon and Paran program of algebraic geometry over the quaternions \cite{AlonParan:2021, AlonParan:2024}, although the corollary does not appear in their papers. 	
\end{rem}

\section{Fixed points and periodic points} \label{sec:dynamics}

We switch back to an associative division ring $D$. We say that $b\in D$ commutes with $\vec{a}\in D^n_c$ if $(\vec{a},b)\in D^{n+1}_c$. 

\begin{lem}
	Let $D$ be a division ring, and let $f,g\in D[x_1,\ldots,x_n],$ and $\vec{a}\in D^n_c$. If $g(\vec{a})$ commutes with $\vec{a}$ and $h=fg$, then $h(\vec{a})=f(\vec{a})g(\vec{a})$. 
\end{lem}

\begin{proof}
	The vector $\vec{a}$ is a root of $g-g(\vec{a})$. Therefore, by Theorem \ref{thm:root-of-left-multiple}, we know $\vec{a}$ is a root of $f\left(g-g(\vec{a})\right) = h - fg$ too. 
	
	Note that since $\vec{a}$ commutes with $g(\vec{a})$ we have $\left(f\cdot g(\vec{a})\right)(\vec{a})=f(\vec{a})g(\vec{a})$. Indeed, let $f = \sum_{\vec{d}\in I} c_{\vec{d}} \, x_1^{d_1}\dots x_n^{d_n}$, then $f\cdot g(\vec{a}) = \sum_{\vec{d}\in I} c_{\vec{d}} \, g(\vec{a})x_1^{d_1}\dots x_n^{d_n}$, and so $\left(f\cdot g(\vec{a})\right)(\vec{a})= \sum_{\vec{d}\in I} c_{\vec{d}} g(\vec{a})a_1^{d_1}\dots a_n^{d_n}$ and since $\vec{a}$ commutes with $g(\vec{a})$ this is equal to $\sum_{\vec{d}\in I} c_{\vec{d}}\, a_1^{d_1}\dots a_n^{d_n} g(\vec{a})=f(\vec{a})g(\vec{a})$.
	
	Since the subsitution map is additive, we get $h(\vec{a})-f(\vec{a})g(\vec{a})=0$, which implies  $h(\vec{a})=f(\vec{a})g(\vec{a})$ as required. 
\end{proof}

\begin{cor}\label{cor:product}
If $D$ is a division ring, $f_1,\ldots,f_t \in D[x_1,\ldots,x_n],$ $\vec{a}\in{D^n_c}$ and $f_i(\vec{a})$ commutes with $\vec{a}$ for $1\le i\le t-1$, then $(f_t\cdots f_1)(\vec{a})=f_t(\vec{a})\cdots f_1(\vec{a})$ for any integer $t \ge 2$.
\end{cor}

\begin{cor}\label{cor:power}
If $D$ is a division ring, $f \in D[x_1,\ldots,x_n],$ $\vec{a}\in{D^n_c}$ and $f(\vec{a})$ commutes with $\vec{a}$, then $f^t(\vec{a})=(f(\vec{a}))^t$ for any integer $t \ge 1$.
\end{cor}

Both corollaries are easily proved by induction on $t$.

Given an $n$-tuple $T=(f_1,\dots,f_n)$ of polynomials $f_1,\dots,f_n \in D[x_1,\dots,x_n]$, we define the ``orbit" of $\vec{a}\in D^n$ under $T$ to be the sequence $(\vec{a_i})_{i=1}^\infty$ of vectors in $D^n$ given by $\vec{a}_1=\vec{a}$ and $\vec{a}_{i+1}=T(\vec{a}_i)=(f_1(\vec{a}_i),\dots,f_n(\vec{a}_i))$ for all $i\in \mathbb{N}$.
We say that $\vec{a}\in D_c^n$ ``commutes with its orbit'' if $(\vec{a}_1,\dots,\vec{a}_m) \in D_c^{mn}$ for all $m\in \mathbb{N}$. In particular, if $\vec{a}\in D_c^n$ is a ``fixed point" of $T$, i.e., $T(\vec{a})=\vec{a}$, then it commutes with its orbit.

\begin{thm}
	Let $T=(f_1,\ldots,f_n)$ where $f_1,\ldots,f_n\in D[x_1,\ldots,x_n].$  Suppose that $\vec{a}\in D^n_c$ commutes with its orbit. Then $T^{\circ{m}}(\vec{a})=\vec{a}_{m}$ for any integer $m\ge 1$. 
\end{thm}

\begin{proof}
	For any integer $m\ge 1$ we denote 
	\[T^{\circ{m}}=\left(F_1^{(m)},\ldots,F_n^{(m)}\right),\]
	for appropriate $F_1^{(m)},\ldots,F_n^{(m)} \in D[x_1,\ldots,x_n]$. Clearly, $F_i^{(1)}=f_i$ for $1\le i \le n$, and
	\[F_i^{(m)}=f_i(T^{\,\circ (m-1)})=f_i\left(F_1^{(m-1)},\ldots,F_n^{(m-1)}\right).\]
	
	Set $\vec{b}=\vec{a}_{m-1}$ and $\vec{e}=\vec{a}_m$. We prove by induction on $m$ that $F_i^{(m)}(\vec{a})=e_i$ for any integer $m\ge{1}$ and any $1\le{i}\le{n}$. This is trivial for $m=1$. Assume $F_i^{(m-1)}(\vec{a})=b_i$. Write 
	\[f_i=\sum_{\vec{d} \in I_i} c_{\vec{d}} \, x_1^{d_1}\dots x_n^{d_n}\]

Now, 
	\[F_i^{(m)} = \sum_{\vec{d} \in I_i} c_{\vec{d}} \, \left(F_1^{(m-1)}\right)^{d_1}\dots \left(F_n^{(m-1)}\right)^{d_n}.\]
Since $\vec{a}$ commutes with $F_i^{(m-1)}(\vec{a})=b_i$ for $1\le{i}\le{n}$, using the additivity of the substitution map and Corollaries \ref{cor:product} and \ref{cor:power}, we get
\begin{align*}
	F_i^{(m)}(\vec{a}) &= \sum_{\vec{d} \in I_i} c_{\vec{d}} \, F_1^{(m-1)}(\vec{a})^{d_1}\dots F_n^{(m-1)}(\vec{a})^{d_n} \\
		&= \sum_{\vec{d} \in I_i} c_{\vec{d}} \, b_1^{d_1}\dots b_n^{d_n} = f_i(\vec{b})=e_i,
\end{align*}
as required. 

Therefore for any integer $m\ge{1}$ we have $T^{\,\circ{m}}(\vec{a})=\left(F_1^{(m)}(\vec{a}),\ldots,F_n^{(m)}(\vec{a})\right)=(e_1,\ldots,e_n)=\vec{e}.$
\end{proof}

\begin{cor}
Let $T=(f_1,\ldots,f_n)$ where $f_1,\ldots,f_n\in D[x_1,\ldots,x_n].$  Suppose that $\vec{a}\in D^n_c$ is a fixed point of $T$, i.e.\ $T(\vec{a})=\vec{a}$. Then $T^{\circ{m}}(\vec{a})=\vec{a}$ for any integer $m\ge 1$. 
\end{cor}

\begin{cor}
Let $T=(f_1,\ldots,f_n)$ where $f_1,\ldots,f_n\in D[x_1,\ldots,x_n],$ $\vec{a}\in D^n_c$ and let $(\vec{a_i})_{i=1}^\infty$ be the orbit of $\vec{a}$. If $a_1 = a_m$ for some integer $m \ge 1$ (i.e., the orbit is periodic), then $T^{\circ{mn}}(\vec{a})=\vec{a}$ for any integer $n\ge 1$. 
\end{cor}

\bibliographystyle{abbrv}
\bibliography{bibfile}

@book {GilleSzamuely2006,
    AUTHOR = {Gille, Philippe and Szamuely, Tam{\'a}s},
     TITLE = {Central simple algebras and {G}alois cohomology},
    SERIES = {Cambridge Studies in Advanced Mathematics},
    VOLUME = {101},
 PUBLISHER = {Cambridge University Press, Cambridge},
      YEAR = {2006},
     PAGES = {xii+343},
      ISBN = {978-0-521-86103-8; 0-521-86103-9},
   MRCLASS = {16K20 (14F22 19C30)},
  MRNUMBER = {2266528},
MRREVIEWER = {Gr{\'e}gory Berhuy},
       DOI = {10.1017/CBO9780511607219},
       URL = {http://dx.doi.org/10.1017/CBO9780511607219},
}

@article {AlonParan:2021,
    AUTHOR = {Alon, Gil and Paran, Elad},
     TITLE = {A central quaternionic {N}ullstellensatz},
   JOURNAL = {J. Algebra},
  FJOURNAL = {Journal of Algebra},
    VOLUME = {574},
      YEAR = {2021},
     PAGES = {252--261},
      ISSN = {0021-8693,1090-266X},
   MRCLASS = {16H05 (16S36)},
  MRNUMBER = {4213620},
MRREVIEWER = {John\ S.\ Kauta},
       DOI = {10.1016/j.jalgebra.2021.01.018},
       URL = {https://doi-org.bengurionu.idm.oclc.org/10.1016/j.jalgebra.2021.01.018},
}

@article {AlonParan:2024,
    AUTHOR = {Alon, Gil and Paran, Elad},
     TITLE = {On the geometry of zero sets of central quaternionic
              polynomials},
   JOURNAL = {J. Algebra},
  FJOURNAL = {Journal of Algebra},
    VOLUME = {659},
      YEAR = {2024},
     PAGES = {780--788},
      ISSN = {0021-8693,1090-266X},
   MRCLASS = {16H05 (14A22 16S36)},
  MRNUMBER = {4779787},
MRREVIEWER = {Liyu\ Liu},
       DOI = {10.1016/j.jalgebra.2024.07.019},
       URL = {https://doi-org.bengurionu.idm.oclc.org/10.1016/j.jalgebra.2024.07.019},
}

@article{AlonChapmanParan:2025,
author = {Alon, Gil and Chapman, Adam and Paran, Elad},
title = {On the geometry of zero sets of central quaternionic polynomials. {II}},
YEAR = {2025},
journal = {Israel Journal of Mathematics, to appear},
doi = {10.1007/s11856-025-2879-y}
}

@article {Chapman:2020,
    AUTHOR = {Chapman, Adam},
     TITLE = {Polynomial equations over octonion algebras},
   JOURNAL = {J. Algebra Appl.},
  FJOURNAL = {Journal of Algebra and its Applications},
    VOLUME = {19},
      YEAR = {2020},
    NUMBER = {6},
     PAGES = {2050102, 10},
      ISSN = {0219-4988},
   MRCLASS = {17D05 (15A18 15B33)},
  MRNUMBER = {4120079},
       DOI = {10.1142/S0219498820501029},
       URL = {https://doi.org/10.1142/S0219498820501029},
}

@article{Chapman:2020b,
 author = {Chapman, Adam},
 title = {Factoring octonion polynomials},
 fjournal = {International Journal of Algebra and Computation},
 journal = {Int. J. Algebra Comput.},
 issn = {0218-1967},
 volume = {30},
 number = {7},
 pages = {1457--1463},
 year = {2020},
 language = {English},
 doi = {10.1142/S0218196720500484},
 zbMATH = {7261096},
 Zbl = {1479.17012}
}

@article {ChapmanMachen:2017,
    AUTHOR = {Chapman, Adam and Machen, Casey},
     TITLE = {Standard polynomial equations over division algebras},
   JOURNAL = {Adv. Appl. Clifford Algebr.},
  FJOURNAL = {Advances in Applied Clifford Algebras},
    VOLUME = {27},
      YEAR = {2017},
    NUMBER = {2},
     PAGES = {1065--1072},
      ISSN = {0188-7009},
   MRCLASS = {16K20 (15A18 15B33)},
  MRNUMBER = {3651503},
MRREVIEWER = {Ivan D. Chipchakov},
       DOI = {10.1007/s00006-016-0740-4},
       URL = {https://doi.org/10.1007/s00006-016-0740-4},
}

@article{ChapmanVishkautsan:2021,
 author = {Chapman, Adam and Vishkautsan, Solomon},
 title = {Fixed points of polynomials over division rings},
 fjournal = {Bulletin of the Australian Mathematical Society},
 journal = {Bull. Aust. Math. Soc.},
 issn = {0004-9727},
 volume = {104},
 number = {2},
 pages = {256--262},
 year = {2021},
 language = {English},
 doi = {10.1017/S0004972721000113},
 zbMATH = {7394391},
 Zbl = {1485.16024}
}

@article{ElduqueVilla:2005,
 author = {Elduque, Alberto and Villa, Oliver},
 title = {A note on the linkage of {Hurwitz} algebras},
 fjournal = {Manuscripta Mathematica},
 journal = {Manuscr. Math.},
 issn = {0025-2611},
 volume = {117},
 number = {1},
 pages = {105--110},
 year = {2005},
 language = {English},
 doi = {10.1007/s00229-004-0512-7},
 zbMATH = {2175526},
 Zbl = {1103.11014}
}

@article {GordonMotzkin1965,
    AUTHOR = {Gordon, B. and Motzkin, T. S.},
     TITLE = {On the zeros of polynomials over division rings},
   JOURNAL = {Trans. Amer. Math. Soc.},
  FJOURNAL = {Transactions of the American Mathematical Society},
    VOLUME = {116},
      YEAR = {1965},
     PAGES = {218--226},
      ISSN = {0002-9947},
   MRCLASS = {12.10 (16.46)},
  MRNUMBER = {0195853 (33 \#4050a)},
}

@book{hartshorne,
  author = {Hartshorne, Robin},
  publisher = {Springer},
  series = {Graduate Texts in Mathematics},
  title = {{Algebraic Geometry}},
  volume = 52,
  year = 1977
}

@book {SpringerVeldkamp,
    AUTHOR = {Springer, Tonny A. and Veldkamp, Ferdinand D.},
     TITLE = {Octonions, {J}ordan algebras and exceptional groups},
    SERIES = {Springer Monographs in Mathematics},
 PUBLISHER = {Springer-Verlag, Berlin},
      YEAR = {2000},
     PAGES = {viii+208},
      ISBN = {3-540-66337-1},
   MRCLASS = {17A75 (17C10 17C30 20G15)},
  MRNUMBER = {1763974},
MRREVIEWER = {Plamen Koshlukov},
       DOI = {10.1007/978-3-662-12622-6},
       URL = {https://doi.org/10.1007/978-3-662-12622-6},
}

@article {Wedderburn:1921,
    AUTHOR = {Wedderburn, J. H. M.},
     TITLE = {On division algebras},
   JOURNAL = {Trans. Amer. Math. Soc.},
  FJOURNAL = {Transactions of the American Mathematical Society},
    VOLUME = {22},
      YEAR = {1921},
    NUMBER = {2},
     PAGES = {129--135},
      ISSN = {0002-9947},
   MRCLASS = {16K40 (12E15)},
  MRNUMBER = {1501164},
       DOI = {10.2307/1989011},
       URL = {https://doi.org/10.2307/1989011},
}

\end{document}